\documentclass[12pt,a4paper]{amsart}
\usepackage{graphicx, amssymb,mathrsfs}
\usepackage{color}
\usepackage[margin=2.5cm]{geometry}
\usepackage[normalem]{ulem}
\newtheorem{thm}{Theorem}[section]
\newtheorem{cor}[thm]{Corollary}
\newtheorem{lem}[thm]{Lemma}

\theoremstyle{definition}

\theoremstyle{remark}

\newtheorem*{rem*}{Remark}
\numberwithin{equation}{section}
\newcommand{\norm}[1]{\left\Vert#1\right\Vert}
\newcommand{\abs}[1]{\left\vert#1\right\vert}
\newcommand{\set}[1]{\left\{#1\right\}}

\newcommand{\lpar}{\left(}
\newcommand{\rpar}{\right)}
\newcommand{\Int}{\operatorname{Int}}

\newcommand{\afterhat}{\hat{\phantom{m}}}

\newcommand{\R}{\mathbb{R}}
\newcommand{\C}{\mathbb{C}}
\newcommand{\N}{\mathbb{N}}
\newcommand{\Z}{\mathbb{Z}}
\newcommand{\Hil}{\mathcal{H}}
\newcommand{\Ban}{\mathcal{X}}
\newcommand{\Lin}{\mathcal{L}}

\newcommand{\cL}{\mathcal{L}}

\newcommand{\frk}[1]{\mathfrak{#1}}

\renewcommand{\Im}{\operatorname{Im}}
\renewcommand{\Re}{\operatorname{Re}}

\newcommand{\Strip}{\mathbb{S}}
\newcommand{\Discr}{\mathbb{D}_m}

\newcommand{\bfc}{\mathbf{c}_m}

\newcommand{\rep}{\rho}

\def\be{\begin{equation}}
\def\ee{\end{equation}}
\def\bes{\begin{equation*}}
\def\ees{\end{equation*}}
\def\bea{\begin{equation}\begin{aligned}}
\def\eea{\end{aligned}\end{equation}}
\def\beas{\begin{equation*}\begin{aligned}}
\def\eeas{\end{aligned}\end{equation*}}

\newcommand{\eb}{\mathbf{e}} 
\newcommand{\var}{s} 
\newcommand{\funz}{g}
\newcommand{\funzz}{h}

\begin{document}

\title[Admissibility of Banach representations]{Harish-Chandra's admissibility problem \\ for Banach space representations of $\mathrm{SL}(2,\R)$}%

\author{Francesca Astengo}
\address{Dipartimento di Matematica, Dipartimento di Eccellenza 2023--2027, Universit\`a di Genova, Via Dodecaneso 35, 16146 Genova, Italy} 
\email{{\tt francesca.astengo@unige.it}}

\author{Michael G. Cowling}
\thanks{The second author was supported by the Australian Research Council, Grant DP220100285.
He also thanks the Universities of Genova and Milano-Bicocca for their hospitality.}
\address{School of Mathematics and Statistics, University of New South Wales, Sydney  NSW 2052, Australia } 
\email{{\tt m.cowling@unsw.edu.au}}

\author{Bianca Di Blasio}
\thanks{}
\address{Dipartimento di Matematica e Applicazioni, Universit\`a di Milano-Bicocca, Via Cozzi 53, 20125 Milano, Italy } 
\email{{\tt bianca.diblasio@unimib.it}}


\begin{abstract}{
We show that irreducible strongly continuous representations of $\mathrm{SL}(2,\mathbb{R})$ on certain Banach spaces are admissible and that the admissibility of Banach space representations of SL(2,R) and the invariant subspace problem are intimately related.}
\end{abstract}

\maketitle
\section{Introduction}

Throughout this paper, $G$ stands for the noncompact semisimple Lie group $\mathrm{SL}(2,\R)$ and $K$ for its compact subgroup $\mathrm{SO}(2)$; we write $k_\theta$ for the usual rotation matrix.
We define a \emph{representation} $\rep$ of $G$ on a Banach space $\Ban$ to be a homomorphism of $G$ into the space $\Lin(\Ban)$ of bounded operators on $\Ban$ with bounded inverse; unless otherwise stated, we assume that $\rep$ is continuous for the strong operator topology on $\Lin(\Ban)$. 
We define $\rep$ to be \emph{irreducible} if $\Ban$ contains no nontrivial \emph{closed} $\rep(G)$-invariant subspaces (this is usually called topologically irreducible).
By standard Fourier series arguments, we may write $\Ban$ as a sum of closed subspaces $\Ban_m$ on which $\rep(k_\theta)$ acts by multiplying by $e^{im\theta}$.
The  irreducible representation $\rep$ of $G$ on the Banach space $\Ban$ is said to be \emph{admissible} if $\dim(\Ban_m) \leq 1$ for all $m$.

For us, \emph{all} Banach spaces are  over the complex field.
We say that a Banach space $\Ban$ has the invariant subspace property (ISP) if $\dim(\Ban) \leq 1$ or if every bounded operator on $\Ban$ has a nontrivial closed invariant subspace.

We are going to prove the following result.

\begin{thm}\label{thm:main}
Let $\rep$ be an irreducible representation of $G$ on a Banach space $\Ban$.
Assume that one of the nontrivial Banach subspaces $\Ban_m$ has the invariant subspace property.
Then all the Banach subspaces $\Ban_m$ have
 the invariant subspace property, and $\rep$  is admissible.
\end{thm}

\begin{cor}\label{cor:Hilbert}
Suppose that Hilbert spaces have the invariant subspace property.
Then every irreducible representation of $G$ on a Hilbert space is admissible.
\end{cor}
\color{black}

We give the proof towards the end of this note.
Our analysis focusses on the Banach algebra $L^1(G,\omega)_m$ of functions $f$ with the invariance and integrability properties
\begin{equation*}
f(k_\theta x k_\phi) = e^{-i m(\theta+\phi)} f(x)
\qquad\text{and}\qquad
\int_G \norm{x}^\alpha \abs{f(x)}  \,dx < \infty
\end{equation*}
(the notation will be explained in Section 2 below, together with other background material).
In Section 3 we review the theory of $m$-spherical functions on $G$ and the associated spherical transformation; our main contribution
is a theorem that ensures that functions with the $K$-equivariance property above and whose $m$-spherical transforms behave well belong to $L^1(G,\omega)_m$.
Section 4 uses spectral theory and the results of Sections 2 and 3 to show that there exists $f_m \in L^1(G,\omega)_{m}$ that generates all $L^1(G,\omega)_{m}$; this key result is in the spirit of work of R. Gangolli~\cite{Gan}. 
In Section 5, we prove the main theorem.
First we use the function $f_m$ of Section 4 and the invariant subspace property to prove that $\dim(\Ban_m) = 1$ and then we show that $\dim(\Ban_n) \leq 1$ for all $n$, that is, $\rep$ is admissible.

In the rest of this introduction, we explain the significance of Theorem 1.1.
In the early period of the development of the representation theory of semisimple Lie groups, between 1945 and 1975, say, representations on Hilbert spaces and on Banach spaces were considered; see, for instance, \cite[Chap 4]{War}.
It was shown that irreducible \emph{unitary} representations on Hilbert spaces are admissible, which implies in particular that they may be effectively studied and classified by passing to the Lie algebra of the group in question, but the situation with Banach representations was unclear, though it was known that there were some pathological examples of reducible Banach representations (see below for an example).
In 1988, W. Soergel \cite{Soergel88} finally constructed an example of an irreducible representation of $G$ on a Banach space that was not admissible.

Soergel's construction used the resolution of the invariant subspace problem.
This problem, whether every nontrivial operator on a complex Banach space of dimension at least $2$ admits a nontrivial invariant subspace, was first resolved in the negative by P. Enflo \cite{Enflo1}, and later shorter solutions were produced by C.J. Read \cite{Read1, Read2}.
All the examples produced were operators on Banach spaces that are rather different from Hilbert spaces, so it was not clear whether Soergel's construction could be applied to Hilbert space representations or not.
What was clear, however, was that if $T$ is a bounded linear operator on a Banach space $\Ban$ with no nontrivial invariant subspaces, then nonadmissible irreducible representations can be constructed on the Banach space $C(G/P,\Ban)$; here $P$ is a minimal parabolic subgroup of $G$.
Hence while we have a relatively simple classification of all the irreducible unitary representations, for general Banach space representations, a similar classification is probably impossible.

The invariant subspace problem has been tormenting functional analysts for a long time.
All finite dimensional Banach spaces and all nonseparable Banach spaces have the ISP,  and in 2011, Argyros and Hayden \cite{AH} gave an example of an infinite dimensional separable  Banach space with the ISP, so our theorem is not content-free.
Recent  preprints by P. Enflo \cite{Enflo2} and by C. W. Neville \cite{Neville} claim that Hilbert spaces have the ISP.
We do not propose to confirm the correctness of these papers here.
Rather, the aim of this note is to show that a positive solution to the problem gives rise to admissible irreducible representations of $G$ on Banach spaces.
Our arguments, coupled with those of Soergel \cite{Soergel88}, imply that the admissibility of representations of $\mathrm{SL}(2,\R)$ and  the invariant subspace problem are intimately related.
It is hardly surprising that Harish-Chandra gave up treating representations on Banach spaces.
\color{black}

Together with our previous paper \cite{ACD}, our result shows that if Hilbert spaces have the ISP, then the irreducible uniformly bounded representations of $G$ on Hilbert spaces may be classified, namely, every such representation is equivalent (via conjugation by a bounded operator with bounded inverse) to one of the irreducible uniformly bounded representations described by R. A. Kunze and E. M. Stein \cite{KS60}, and the Kunze--Stein representations are minimal representatives of the equivalence classes.

\section{Definitions and preliminary results}\label{sec:def-prelim}

The early work on $\mathrm{SL}(2,\R)$ is due to V. Bargmann \cite{B}. 
L. Ehrenpreis and F. Mautner \cite{EMI, EMII, EMIII} used Bargmann's work later to study analysis on $\mathrm{SL}(2,\R)$ in detail.
Most of what we write in this section would have been known to these authors.

\subsection*{The group $G$}
Recall that $G = \mathrm{SL}(2,\mathbb{R})$ and $K = \mathrm{SO}(2)$.
The Iwasawa decomposition is that every $x$ in $G$ may be written uniquely in the form $n_t a_r k_\theta$, where $r, t\in \R$ and $\theta \in \R \mod 2\pi$; here
\[
\begin{gathered}
k_\theta := \begin{pmatrix} \cos\theta & \sin\theta \\ -\sin\theta & \cos\theta \end{pmatrix} = \exp\lpar \theta \begin{pmatrix} 0 & 1 \\ -1 & 0 \end{pmatrix} \rpar
\\
a_r := \begin{pmatrix} e^r &0 \\ 0  & e^{-r} \end{pmatrix} 
= \exp\lpar r \begin{pmatrix} 1 & 0 \\ 0 & -1\end{pmatrix} \rpar
\\
n_t := \begin{pmatrix} 1 & 0 \\ t & 1 \end{pmatrix} = \exp\lpar t \begin{pmatrix} 0 & 0\\ 1 & 0 \end{pmatrix} \rpar \ .
\end{gathered}
\]
According to the Cartan decomposition, every element of $G$ may be written in the form $k_\theta a_r k_\phi$, where $r \geq 0$; if $x \in K$ then $r = 0$ and there are many choices for $k_\theta$ and $k_\phi$; otherwise, the decomposition is unique up to changes in $\theta$ and $\phi$ by adding (or subtracting) $\pi$ to both or $2\pi$ to either.
In these coordinates, the Haar measure on $G$ is given by
\begin{equation*}
dx=2 \sinh(2r) \,dr \,d\varphi \,d\psi.
\end{equation*}
Clearly
\[
2 \sinh(2r)\sim
\begin{cases}4r & \text{as $r\to 0$}
\\[3pt]
e^{2r} & \text{as $r\to \infty$}.
\end{cases}
\]

The  Lie algebra $\frk{g}$ of $G$ is spanned  by the elements $E$, $H$ and $Q$, given by
\begin{equation*} 
E   = \begin{pmatrix}
        0 & 1 \\
        -1 & 0
      \end{pmatrix},
\qquad      
H = \begin{pmatrix}
        1 & 0 \\
        0 & -1 
      \end{pmatrix},
\qquad
Q = \begin{pmatrix}
        0 & 1 \\
        1 & 0 
      \end{pmatrix}   ;
\end{equation*}
we define the Casimir operator $\Omega$ by
\[
4\Omega :=H^2+Q^2-E^2.
\]

\subsection*{Representations of $G$}
Throughout, we use the word \emph{representation} to mean a representation $\rho$ of $G$ on a topological vector space $\mathcal{V}$ by continuous operators whose inverse is also continuous, which, unless otherwise stated,
 is strongly continuous, in the sense that $x_n \to x$ in $G$ implies that $\rep(x_n) \xi \to \rep(x)\xi$ in $\mathcal{V}$ for all $\xi \in \mathcal{V}$.
If $\mathcal{V}$ is a Hilbert space or a Banach space, we call $\rho$ a Hilbert or a Banach representation.
For $x \in G$, we write $\|x\|$ for the norm of $x$ as a linear operator on the euclidean space $\R^2$.
It is evident that $\|x\| = \|x^{-1}\|$ and $\| x_1 x_2\| \leq \| x_1 \| \| x_2\|$ for all $x, x_1, x_2$ in $G$, and that if $x$ has Cartan decomposition $k a_r k'$, then $\|x\| = \|a_r\|= e^r$.
Norms in other spaces are indicated by a subscript, for example, $\norm{\rho(x)}_{\mathcal{L}(\Ban)}$ indicates the norm of $\rho(x)$ in the Banach space of bounded linear operators on the Banach space $\Ban$.

\begin{lem}\label{lem:mod-growth}
Suppose that $\rho$ is a representation of $G$ on a Banach space $\Ban$.
Then there exist constants $\alpha$ in $[0,\infty)$, called the growth rate of $\rep$, and $C$ in $[1,\infty)$ such that
\be\label{eq:mod-growth}
\| \rep(x) \|_{\mathcal{L}(\mathcal{X})} \leq C \|x\|^\alpha
\qquad\forall x \in G.
\ee
\end{lem}

\begin{proof}
Take $\xi$ in $\Ban$. 
Since $x \mapsto \| \rho(x) \xi \|_{\Ban}$ is continuous on $G$, it is bounded on compact sets.
The Banach--Steinhaus theorem implies that $\| \rho(x) \|_{\mathcal{L}(\mathcal{X})}$ is bounded on compact sets in $G$.

If $j \in \N^+$ and $e^{j-1} \leq \|x\| \leq e^j$, then we may write $x =x_1 \dots x_j$, where $\|x_i\| \leq e$; indeed, if $x = k a_r k'$,  then we may take $x_1$ to be $k a_{r/j}$, $x_2$, \dots, $x_{j-1}$ to be $a_{r/j}$, and $x_j$ to be $a_{r/j}k'$.  
Hence
\[
\|\rho(x) \|_{\mathcal{L}(\mathcal{X})} 
\leq \|\rho(x_1) \|_{\mathcal{L}(\mathcal{X})} \dots \|\rho(x_j) \|_{\mathcal{L}(\mathcal{X})} 
\leq \max\{ \|\rho(x)\|_{\mathcal{L}(\mathcal{X})} : \|x\| \leq e\}^{j} ;
\]
the result follows with $C$ equal to $\max\{ \|\rho(x)\|_{\mathcal{L}(\mathcal{X})} : \|x\| \leq e\}$ and $\alpha$ equal to $\log C$.
\end{proof}

We often write inequalities such as \eqref{eq:mod-growth}  in the form ``$\| \rep(x) \|_{\mathcal{L}(\mathcal{X})} \lesssim \|x\|^\alpha$ for all $x \in G$'' when the constant $C$ is not important.

We now consider the restriction to $K$ of a representation $\rho$ of $G$  on a Banach space $\Ban$.
For $m \in \Z$, we define the closed subspace $\Ban_m$ of $\Ban$ by
\[
\Ban_{m} := \{ \xi \in \Ban : \rep(k_\theta)\xi = e^{im\theta}\xi \} ,
\]
and $\mathcal{P}_m :\Ban \to \Ban_m$ by
\begin{equation}
\mathcal{P}_m 
:= \frac{1}{2\pi} \int_{-\pi}^{\pi} e^{-im\theta} \rep(k_\theta) \,d\theta.
\end{equation}
Then $\mathcal{P}_m$ is a bounded projection; indeed,
\[
\norm{\mathcal{P}_m}_{\mathcal{L}(\mathcal{X})} 
\leq \frac{1}{2\pi} \int_{-\pi}^{\pi} \norm{\rep(k_\theta)}_{\mathcal{L}(\mathcal{X})}  \,d\theta 
\leq \max\{ \norm{ \rho(k_\theta) }_{\mathcal{L}(\Ban)} : k_\theta \in K \}.
\]
{
Fourier analysis shows that for all $\xi \in \Ban$,
\[
\xi = \lim_{M \to \infty} \sum_{|m| \leq M}  \frac{M - |m|}{M} \mathcal{P}_m \xi
\]
in norm. 
Hence $\sum_{m} \Ban_m$, the algebraic direct sum of the $\Ban_m$, is dense in $\Ban$.}
If $\rep$ is a representation of $G$ on the Banach space $\Ban$, then the closures of $\sum_{m \in 2\Z} \Ban_m$ and $\sum_{m \in 2\Z+1} \Ban_m$ are closed $\rep(G)$-invariant subspaces of $\Ban$.
Thus for irreducible representations, all $\Ban_m$ are trivial if $m$ is odd or if $m$ is even.

\subsection*{The algebra $L^1(G,\omega)_m$}
Given a space $E(G)$ of functions on $G$, we write $E(G)_m$ for  the subset of $E(G)$ of functions $f$ such that
\begin{equation}\label{eq:invariance}
f(k_\theta x k_\phi) = e^{-im(\theta+\phi)} f(x)
\qquad\forall  k_\theta, k_\phi \in K \quad\forall x \in G.
\end{equation}
The spaces $L^1(G)_{m}$ and $L^2(G)_{m}$ are obviously closed subspaces of $L^1(G)$ and $L^2(G)$.

\begin{lem}
The space $L^1(G)_m$, equipped with convolution and the $L^1(G)$-norm, is a commutative Banach algebra.
\end{lem}

\begin{proof}
See \cite[Lemma 3, p.~21]{Lang} or \cite[Proposition 1, p.~290]{T}. 
\end{proof}

For $\alpha \in [0,\infty)$, let $\omega_\alpha$, or just $\omega$, be the function $x \mapsto \|x\|^\alpha$.
Then  
\[
\omega(x) \in [1, \infty),
\qquad
\omega(x) = \omega(x^{-1})
\qquad\text{and}\qquad 
\omega(x_1x_2) \leq \omega(x_1) \omega(x_2)
\]
for all $x, x_1, x_2 \in G$.
A continuous function satisfying these conditions is called a weight.
We write $L^1(G,\omega)$ for the subset of $L^1(G)$ of functions $f$ such that
\[
\norm{f} _{L^1(G,\omega)} := \int_{G} \abs{ f(x) } \omega(x) \,dx < \infty.
\]
Then $L^1(G,\omega)$ is a subspace of $L^1(G)$ that is closed under convolution, and equipped with its own norm, is a Banach algebra (often called a Beurling algebra).
Indeed, 
\[
\begin{aligned}
\norm{ f * g} _{L^1(G,\omega)} 
&\leq \int_G  \abs{ f} * \abs{g} (x) \omega(x) \,dx \\
&\leq \int_G  \abs{ f(y)} \abs{g (y^{-1}x)} \omega(y) \omega(y^{-1}x) \,dy \,dx \\
&= \norm{f} _{L^1(G,\omega)} \norm{g} _{L^1(G,\omega)} .
\end{aligned}
\]
Evidently $C^\infty_c(G)$ is a subspace of $L^1(G,\omega)$, no matter what weight is considered.
Evidently $L^1(G,\omega)_m$ and $C^\infty_c(G)_m$ are subalgebras of the convolution algebra $L^1(G)_m$.

The continuous multiplicative linear functionals on the commutative algebras $C^\infty_c(G)_m$ were determined by Ehrenpreis and Mautner \cite{EMII}, and are called $m$-spherical functions.
Every continuous multiplicative linear functional on $L^1(G,\omega)_m$ restricts to a continuous multiplicative linear functional on the dense subspace $C^\infty_c(G)_m$, and hence is the continuous extension to $L^1(G,\omega)_m$ of one of  the known $m$-spherical functions.
The multiplicative linear functionals on the Banach algebras $L^1(G)_m$, with no weight, were first studied by Ehrenpreis and Mautner \cite[Lemma 7.2]{EMIII}, and further analysed in~\cite{RW, T}. 
We will recall some properties of $m$-spherical functions and extend some of the results of these authors to $L^1(G,\omega)_m$ in Section~3.

The algebras $L^1(G,\omega)_m$ do not have an identity, but the larger space $M(G,\omega)_m$ of suitably bounded complex measures with the (measure theoretic version of the) invariance property \eqref{eq:invariance} contains the measure $\mu_m$:
\[
\mu_m(f) = \frac{1}{2\pi} \int_{-\pi}^{\pi} f(k_\theta) e^{-im\theta} \,d\theta,
\]
which has the properties  that $\mathcal P_m =\rep (\mu_m)$ and
\[
\mu_m *f = f* \mu_m = f 
\qquad\forall f \in L^1(G,\omega)_m.
\]
If $(f_n)_{n\in\N}$ is an approximate identity in the algebra $C^\infty_c(G)$, then $(\mu_m*f_n*\mu_m)_{n\in\N}$ is an approximate identity in the algebra $L^1(G,\omega)_m$ that consists of smooth compactly supported functions.
By considering approximate identities, we may show that when $g \in L^1(G,\omega)_m$, the norm of the operator $f \mapsto g * f$ on $L^1(G,\omega)_m$ is equal to $\norm{g}_{L^1(G,\omega)}$.

\subsection*{Irreducible representations}
By Lemma \ref{lem:mod-growth}, a Banach representation $\rep$ of $G$ has a growth rate $\alpha$.
We take $\omega(x)$ to be $\|x\|^\alpha$, and then $\rep$ induces a representation, also denoted by $\rho$, of $L^1(G,\omega)_{m}$ on $\Ban_{m}$ by integration.
Since $L^1(G,\omega)_{m}$ contains an approximate identity of smooth functions, the subspace $\Ban_m^\infty$ of all smooth vectors in $\Ban_m$ is dense in $\Ban_m$.
The space $\Ban_m^\infty$ is also the subspace of $\Ban^\infty$, the space of smooth vectors in $\Ban$, of vectors satisfying \eqref{eq:invariance}.

\begin{lem}\label{lem:irred-on-L1m}
Let $\rho$ be an irreducible representation of $G$ on the Banach space $\Ban$.
Then the restriction of $\rho$ to $L^1(G,\omega)_m$ is an irreducible representation of $L^1(G,\omega)_m$  on $\Ban_m$.
\end{lem}

\begin{proof}
To prove our claim, it suffices to take $\xi$ in $\Ban_m\setminus\{0\}$ and $\eta$ in $\Ban_m$ and show that there is a sequence $(h_n)_{n\in\N}$ in $L^1(G,\omega)_m$ such that $\rho(h_n) \xi \to \eta$ as $n \to \infty$.

Since $\rho$ is an irreducible representation of $G$, its integrated version is an irreducible representation of $L^1(G,\omega)$ and there exists a sequence $(g_n)_{n\in\N}$ in $L^1(G,\omega)$ such that $\lim_n \rep(g_n) \xi = \eta$ in $\Ban$.
Then $h_n := \mu_m * g_n * \mu_m \in L^1(G,\omega)_m$, and
\[
\eta 
=\mathcal{P}_m \eta 
= \lim_n \mathcal{P}_m \rep(g_n) \xi 
= \lim_n \mathcal{P}_m \rep(g_n) \mathcal{P}_m \xi 
= \lim_n  \rep(\mu_m * g_n * \mu_m) \xi ,
\]
so we are done.
\end{proof}

\subsection*{Admissible representations}
The representation $\rep$ of $G$ gives rise to a representation of $\frk{g}$ on $\sum_{m\in\Z} \Ban_m^\infty$ (the algebraic sum): for $X \in \frk{g}$ and $\xi \in \sum_{m\in\Z} \Ban_m^\infty$,
\[
\rep(X)\xi = \frac{d}{dt} \rep(\exp(tX)) \xi |_{t=0} ;
\]
we extend $\rep$ to a representation of $\frk{g}$ by linearity.

We say that $\rep$ is an admissible representation when each $\Ban_m$ is finite dimensional. 
This means that $\Ban_m$ coincides with $\Ban_m^\infty$, and the algebraic direct sum $\sum_{m\in \Z} \Ban_m$ is then a space on which both $\frk{g}$ and $K$ act, with an obvious compatibility condition between the two actions.
When an additional finite generation condition is satisfied, which is automatic for an irreducible representation, such a space is known as a  Harish-Chandra module.
Conversely, every Harish-Chandra module arises from a representation (though this representation need not be irreducible).  
Indeed, the Casselman--Wallach theorem shows that there is always a continuous Fr\'echet space representation of $G$ associated to any Harish-Chandra module, and P. Delorme and S. Souaifi \cite{DS} later showed that we may take the Fr\'echet space to be a Hilbert space; see \cite{BeKr14} for more information.

\subsection*{The Plancherel formula}
Irreducible \emph{unitary} representations are admissible~\cite[\S 4.5]{War}. 
If $\pi$ is an irreducible unitary representation on the Hilbert space $\Hil^\pi$, then the subspace $\Hil^\pi_{m}$ of vectors which transform  under the action of $K$ according to the character $k_\theta\mapsto e^{im\theta}$ is at most one-dimensional. 
We denote by $\eb^\pi_{m}$ a unit vector in $\Hil^\pi_{m}$ if this exists. 


The only representations appearing in the Plancherel formula are the so-called unitary principal series and discrete series representations. We now fix notation for these. 

Following \cite{ACD,BCNT}, we write the unitary principal series representations as $\pi_{i\lambda,\epsilon}$, where $\lambda\in \mathbb R$ and $\epsilon=0,1$.  
The representation  $\pi_{i\lambda,\epsilon}$ is unitarily equivalent to $\pi_{-i\lambda,\epsilon}$, and is  irreducible except  when $\lambda =0$ and $\epsilon=1$, when the closures of the subspaces $\sum_{m>0} \Ban_m$ and $\sum_{m< 0} \Ban_m$ are invariant subspaces on which $\pi_{0,1}$ acts irreducibly;
as $\pi_{0,1}$ has Plancherel measure $0$, this is of no consequence in what follows.
When $\pi=\pi_{i\lambda,\epsilon}$, for every $m\in \Z$, the subspace $\Hil^{\pi}_m$ is one-dimensional if $m-\epsilon$ is even and trivial otherwise.


The discrete series of representations consists of two contragredient families, $\pi_\var^+$ and  $\pi_\var^-$, where $\var\in \frac12\mathbb Z\setminus\{0\}$. 
The representations $\pi_\var^\pm$ and $\pi_{-\var}^\pm$ are unitarily equivalent.  
For the reader's convenience, we note that $\pi_{\var}^\pm$ in our notation corresponds to $\pi_h^\mp$  in the notation of~\cite{BCNT} where  $h=|\var|+1/2 $.
When $\pi=\pi^\pm_s$, the subspace $\Hil^{\pi}_m$  is one-dimensional if $\pm m\in 2|\var| +1 + 2\mathbb N $, and trivial otherwise. 

The Plancherel formula is
\begin{equation}\label{e:plancherel}
\begin{aligned}
\|f\|_{L^2(G)}^2&=\frac{1}{8\pi} \int_\R \|\pi_{i\lambda,0}(f)\|_{HS}^2 \lambda\tanh(\pi\lambda)\, d\lambda
+
 \frac{1}{8\pi} \int_\R \|\pi_{i\lambda,1}(f)\|_{HS}^2 \lambda\coth(\pi\lambda)\, d\lambda
 \\
 &+\frac{1}{8\pi} \sum_{s\in \frac12 \Z \setminus\{0\}} |s| \left(\|\pi_s^+(f)\|_{HS}^2 +\|\pi_s^-(f)\|_{HS}^2 \right) ,
\end{aligned}
\end{equation}
where $\norm{\cdot}_{HS}$ denotes the Hilbert--Schmidt norm.
%
 
 \color{black}
\subsection*{Pathological representations}
There are natural representations of $G$ acting on $C(T)$ and $L^\infty(T)$, the spaces of continuous and bounded measurable functions on the circle $T$.
It is easy to construct pathological representations of $G$, which are weakly but not strongly continuous, by considering the quotient representation on $L^\infty(T)/C(T)$.

The next lemma shows that irreducible Banach representations are not too pathological.

\begin{lem}
Suppose that $\rep$ is an irreducible representation of $G$ on a Banach space $\Ban$.
Then $\Ban$ is separable.
\end{lem}

\begin{proof}
Take $\xi \in \Ban$.
The closure of the $G$-invariant subspace $\{\rep(f) \xi : f \in L^1(G,\omega) \}$ of $\Ban$ is $G$-invariant and separable, and must coincide with $\Ban$ if $\rep$ is irreducible.
\end{proof}

\section{The generalised spherical transformation}

{
In this section we describe some well known facts about the Fourier transform of $m$-type functions (see for example \cite{B, Koo, RW, T}). 
Then, in Theorem \ref{thm:PW}, we recall the Paley--Wiener theorem for $C^\infty_c(G)_m$, in Theorem~\ref{thm:f-star-g} we link different $m$-types, 
and in Theorem~\ref{molt1} we establish a criterion for a function to be the Fourier transform of an $L^1(G, \omega)_m$ function.
}

We consider the subalgebra $C^\infty_c(G)_m$ of the commutative Banach algebra $L^1(G)_m$, where $m \in \Z$.
 The multiplicative linear functionals on $C^\infty_c(G)_m$ are given by integration against generalised spherical functions. 
For $\var \in \C$, we define the (generalised) $m$-spherical function $\phi_{m,\var}$ by
\[
\phi_{m,\var}(x) :=  \frac{1}{2\pi} \int_{-\pi}^{\pi}  f_{m,\var}(k_\theta ^{-1}xk_\theta)\, d\theta ,
\]
where $f_{m,\var}(k_\theta a_r n_t)=e^{-im\theta}\, e^{r(1-2\var)}$ defines $f_{m,s}$ in terms of the Iwasawa decomposition. 
%
It is easily checked that $|\phi_{m,\var}|\leq \phi_{0, \left|\Re\var\right|}$, and that
\[
\Omega \phi_{m,\var}=\left(\var^2-\tfrac14\right) \phi_{m,\var} .
\]
The (generalised) $m$-spherical transform of $f$ in $C^\infty_c(G)_m$ is the 
function $\hat f$ defined by
\begin{equation*}
\hat f(\var) :=\int_G f(x)\, \phi_{m,\var}(x^{-1})\, dx 
\qquad\forall \var\in \C.
\end{equation*}
Spherical functions corresponding to the same eigenvalues are equal, so $\phi_{m,\var}=\phi_{m,-\var}$,  and therefore $\hat f(\var)=\hat f(-\var)$.
It has been observed by many authors \cite{B,Koo,RW} that the $m$-spherical functions may be expressed in terms of hypergeometric functions or Jacobi functions. Indeed,
\begin{equation}\label{eq:hypergeometric}
\begin{aligned}
\phi_{m,\var}(a_r)&=(\cosh r)^{-m}\, {{}_2F_1}\left(\tfrac{1-m}{2}-\var,\tfrac{1-m}{2}+\var;1;-\sinh^2 r\right)
\\
&=(\cosh r)^{-1-2\var}\, {{}_2F_1}\left(\tfrac{1+m}{2}+\var,\tfrac{1-m}{2}+\var;1;\tanh^2 r\right)
\\
&=(\cosh r)^{-1+2\var}\, {{}_2F_1}\left(\tfrac{1-m}{2}-\var,\tfrac{1+m}{2}-\var;1;\tanh^2 r\right).\end{aligned}
\end{equation}
These identities imply that 
$\phi_{m,\var}(a_r)=\phi_{-m,\var}(a_r)=\phi_{m,-\var}(a_r)$.

We write $\N := \{ 0,1,\ldots\}$, and when $\delta\ge 0$, we define the closed vertical strip $\Strip_\delta$ and the finite set $\Discr$  in the complex plane by
\[
\Strip_\delta := \set{\var\in \C\, :\, \abs{ \Re( \var) } \le \delta/2}
\]
and
\[
\Discr := \set{\var\in \R\setminus\{0\}\, :\,  |m| - 1 - 2|\var| \in 2\N }.
\]
When $\var \in \Discr$, the hypergeometric function in the second or third (depending on the sign of $\var$) line of \eqref{eq:hypergeometric} is a polynomial
 of degree $(|m|-1 - 2|\var|)/2$.
Moreover, as $r\to \infty$,
\[
|\phi_{m,\var}(a_r)|
\lesssim 
\begin{cases}(1+r)\,e^{-r} &\text{ if } \var=0
\\[3pt]
e^{(2\left|\Re(\var) \right|-1)r} & \text{ if } \var\in\C\setminus (\Discr \cup \{0\}) 
\\[3pt]
e^{(-2|\var|-1)r}&\text{ if } \var\in \Discr\ .
\end{cases}
\]
Take $\omega(x)$ to be $\|x\|^\alpha$.
Then $x\mapsto |\phi_{m,\var}(x)|\, \omega^{-1}(x)$ is bounded when $\var$ is in $\Strip_{\alpha+1}\cup\Discr$ and therefore we can extend the map $f\mapsto \hat f(\var)$
continuously to $L^1(G,\omega)_m$.
If $f \in C_c^\infty(G)_m$ 
has support in $\{ x \in G: \norm{x} \leq R\}$, then $\hat f$ is an even entire function of exponential type $2R$, that is, 
\be\label{stimastriscia}
|\hat f(s)|  \lesssim \exp(2R |\Re(s)|) 
\qquad\forall s \in \C .
\ee

\bigskip

If $\var\in \Discr$, then $\phi_{m,\var}\in L^2(G)_m$,
and if $\var\in \Discr\setminus\Strip_{\alpha+1}$, then $\phi_{m,\var}\in L^1(G,\omega)_m$.
Indeed,
\be\label{e:stimaphi}
\|\phi_{m,\var}\|_{L^1(G,\omega)}
\lesssim \int_0^\infty e^{(-2|\var|-1)r}\,e^{\alpha r}\,e^{2r}\, dr \leq \frac1{2|\var|-\alpha-1}.
\ee
The asymptotic behaviour of the spherical functions for large $r$ was studied in~\cite[p.~53]{Koo} and \cite[Lemma~4.7]{RW}.
Define  the (generalised) Harish-Chandra function $\bfc$ by the rule
\begin{align*}
&\bfc(\var) := \frac{2^{1-2\var}\,
\Gamma\left(2\var\right)
}{
\Gamma\left(\frac{1+2\var+m}{2}\right)\Gamma\left(\frac{1+2\var-m}{2}\right)
}\ .
\end{align*}
Then the function $\bfc^{-1}$ is meromorphic with simple poles at the points $\var$ with
\[
2\var\in 
\begin{cases} 
-|m|-1-2\N &\text{when $s < 0$} 
\\[3pt]
 \ \ |m|-1-2\N &\text{when $s > 0$}\, ,
 \end{cases}
\]
and
\[
|\bfc(i\lambda)|^{-2}
=
\begin{cases}
\pi\lambda  \tanh (\pi\lambda) &\text{if $m$ is even}
\\[3pt]
\pi\lambda \coth (\pi\lambda)&\text{if  $m$  is odd}.
\end{cases}
\]
\color{black}
Moreover
\be\label{stimaHC}
|\bfc(\var)^{-1} |\sim |\var|^{1/2} 
\qquad \text{as $\var \to \infty$ in $\Strip_\delta$}.
\ee
When $2\var$ is not an integer and $r > 0$, using \cite[2.10 (3)]{ErdI}, we may write 
\be\label{PHI}
\phi_{m,\var}(a_r)=\mathbf{c}_m(\var)\,\Phi_{m,\var}(a_r)+\mathbf{c}_m(-\var)\,\Phi_{m,-\var}(a_r) ,
\ee
where 
\begin{align*}
\Phi_{m,\var}(a_r)&= (2\cosh r)^{2\var-1}\,
{{}_2F_1}\left(\tfrac{1-m}{2}-\var,\tfrac{1+m}{2}-\var;1-2\var;\cosh^{-2} r\right) 
\\&=e^{(2\var-1)r}(1+\nu_m(\var,r)) ;
\end{align*}
the function $\var\mapsto \nu_m(\var,r)$ is holomorphic in $\Re(\var)<1$ and, for all fixed $\sigma>0$, satisfies 
\be\label{stima_nu}
 \sup_{\Re (\var)\leq 0}|\nu_m(\var,r)|\leq C_m  \qquad  \forall r\geq \sigma.
\ee

The positive-definite $m$-spherical functions were determined in~\cite{T}.
These are the $m$-spherical functions with $\var$ in $i \R $, corresponding to the principal series representations, and $m$-spherical functions with $\var$ in $\Discr$,  corresponding to the discrete series representations; indeed,
\begin{equation}\label{e:sfer-repr}
\phi_{m,\var}(x)=\langle \pi(x)\eb^{\pi}_m,\eb^{\pi}_m\rangle ,
\end{equation}
where
$\pi= \pi_{\var,\epsilon}$ for $\var\in i\R$ and $m \in 2\Z + \epsilon$, or
$\pi=  \pi_{\var}^\pm$ for $s\in \Discr$ and $\pm m \in 2|s|+ 1 + 2\N$.

The following inversion formulas for $f$  in $C^\infty_c(G)_m$ are proved in~\cite{T} and \cite{Koo},
and are a variant of the Plancherel formula \eqref{e:plancherel}:
\color{black}
\begin{equation}\label{e:inv2}
f(x)
=\frac{1}{8\pi^2}\int_\R \hat{f}(i\lambda)\,\phi_{m,i\lambda}(x)\,|\bfc(i\lambda)|^{-2}
\, d\lambda
+\frac{1}{8\pi}\sum_{\var\in \Discr}|\var|\, \hat{f}(\var)\, \phi_{m,\var}(x);
\end{equation}
and if   $\delta\geq 0$, $\delta\notin \Discr$, $x\neq e$, then
\begin{equation}\label{e:inv3}
f(x)=\frac{1}{4\pi^2}\int_\R \hat{f}(\delta+i\lambda)\,
\mathbf{c}_m(\delta+i\lambda)^{-1}\, \Phi_{m,-\delta-i\lambda}(x) \, d\lambda 
+\frac{1}{8\pi}\sum_{\var\in \Discr\setminus \Strip_{2\delta}}|\var|\, \hat{f}(\var)\, \phi_{m,\var}(x).
\end{equation}

Formula  \eqref{e:inv2}, which is well known, can be easily extended to $f\in L^1(G)_m$ such that 
\[
\int_\R |\hat f (i\lambda)|\,|\bfc(i\lambda)|^{-2}\,d\lambda<\infty .
\]
 Koornwinder proved formula \eqref{e:inv3}  for $\delta>m/2$  and stated that \eqref{e:inv2} follows.
 We show how \eqref{e:inv2} implies \eqref{e:inv3} because  we need this argument for the  proof of Theorem~\ref{molt1}.
 
By definition, $\mathbf{c}_m(-i\lambda)=\overline{\mathbf{c}_m(i\lambda)}$ 
 for all $\lambda\in \R$. 
Since $\hat f$ is even, when $r>0$, \eqref{PHI} and \eqref{e:inv2} imply that
\begin{align*}
f(a_r)
&=\frac{1}{4\pi^2}\int_\R\hat f(i\lambda)\,\mathbf{c}_m(-i\lambda)\,\Phi_{m,-i\lambda}(a_r)
\,|\mathbf{c}_m(i\lambda)|^{-2}
\, d\lambda 
+\frac{1}{8\pi}\sum_{\var\in \Discr}|\var|\, \hat f(\var)\, \phi_{m,\var}(a_r)
\\
&
=\frac{1}{4\pi^2}\int_\R\hat f(i\lambda)\,
\mathbf{c}_m(i\lambda)^{-1}\, \Phi_{m,-i\lambda}(a_r) \, d\lambda 
+\frac{1}{8\pi}\sum_{\var\in \Discr}|\var|\, \hat f(\var)\, \phi_{m,\var}(a_r).
\end{align*}

 Let  $\delta\in (0, +\infty)\setminus\Discr$. 
We shift the contour of integration, taking the simple poles of the function  $\bfc^{-1}$, whose residue at $\var_j=\frac{1}{2}( |m|-1 )-j>0$ is $(-1)^j\binom{|m|-j}{j} 2^{-2+|m|-2j}$, into account. 
Combining this fact  with   \cite[10.8 (16)]{ErdII},  we obtain
\[
\mathrm{Res}(\bfc^{-1}(\var)\Phi_{m,-\var},\var=\var_j)=\frac{\var_j}{2}\,\phi_{m,\var_j}\ .
\]
Applying Cauchy's residue theorem gives \eqref{e:inv3}.

\bigskip


\bigskip

Using the facts above, Koornwinder proved the following Paley--Wiener Theorem~\cite[Theorem 2.1]{Koo}.

\begin{thm}\label{thm:PW}
The (generalised) $m$-spherical transform is bijective  from  $C^\infty_c(G)_m$ onto the space of even entire functions $\psi$ for which there exists a positive constant $C_\psi$ such that
\begin{equation}\label{e:uet}
\sup_{\var\in \mathbb C} (1+|\var|)^{N}\, |\psi(\var)| e^{-C_\psi |\mathrm{Re} \var|}<\infty.
\end{equation}
 for every positive integer $N$.
\end{thm}

Note that the conditions in Theorem~\ref{thm:PW} do not depend on $m$. This fact allows us to 
find a link between the various $m$-types.
%
%

\begin{thm}\label{thm:f-star-g}
Fix $m,n \in \Z$ such that $m-n$ is even.
Suppose that $g \in C^\infty_c(G)$ and 
\begin{equation}\label{eq:invariance-mn}
g(k_\theta x k_\phi) = e^{-i n\theta-im\phi}  g(x) 
\qquad\forall x \in G  \quad\forall \theta, \phi \in \R.
\end{equation}
Then for all $f \in C^\infty_c(G)_n$, there exists $f' \in C^\infty_c(G)_m$ such that 
\[
f * g = g * f'  .
\]
\end{thm}

\begin{proof}
By~\eqref{eq:invariance-mn}, if $\pi$ is a unitary representation, then
\[
\langle \pi(g)\eb^\pi_h,\eb^\pi_j\rangle=\int_G g(x)\, \langle \pi(x)\eb^\pi_h,\eb^\pi_j\rangle\, dx=0
\]
unless  $j=n$ and $h=m$. Moreover, 
%
\[
\langle \pi(g)\eb^\pi_m,\eb^\pi_n\rangle=0 \quad
\text{if}\quad
\begin{cases}
\pi=\pi_\var^+ \text{ and } 2|\var|\not\in \min\{m,n\}-1-2\mathbb N
\\[3pt]
\pi=\pi_\var^- \text{ and } 2|\var|\not\in\min\{-m,-n\}-1-2\mathbb N,
\end{cases} 
\]
so that, in the Plancherel formula of a function 
satisfying~\eqref{eq:invariance-mn}, only a finite number of half-integers $\var$ appear.

As  $f \in C^\infty_c(G)_n$, its $n$-spherical transform $\psi$ is an even entire function satisfying~\eqref{e:uet}. 
By Theorem~\ref{thm:PW}, the function $f'$ whose $m$-spherical transform is $\psi$
is in $C^\infty_c(G)_m$. 

We now check that $f*g=g*f'$. By the Plancherel formula~\eqref{e:plancherel} and 
since  $f*g$ and $g*f'$ both satisfy~\eqref{eq:invariance-mn}, it is enough to check that
$\langle \pi(f*g)\eb^\pi_m,\eb^\pi_n\rangle=\langle \pi(g*f')\eb^\pi_m,\eb^\pi_n\rangle$
in the following cases:
\begin{itemize}
\item $\pi=\pi_{i\lambda,\epsilon}$, where $\lambda\in \R $
and $m-\epsilon \in 2\Z$ (hence $n-\epsilon$ is even);
\item 
$\pi=\pi^+_s$, where $2|s|\in \min\{m,n\}-1-2\mathbb N$, that is $m,n>0$ and $s\in \Discr\cap{\mathbb{D}_n}$;
\item $\pi=\pi^-_s$, where $2|s| \in\min\{-m,-n\}-1-2\mathbb N$, that is $m,n<0$ and $s\in \Discr\cap{\mathbb{D}_n}$.
\end{itemize}

By formula~\eqref{e:sfer-repr}, 
in all these cases,  \[
\langle \pi(f')\eb^\pi_m, \eb^\pi_m\rangle=
\langle \pi(f)\eb^\pi_n, \eb^\pi_n\rangle,
\]
so that 
\begin{align*}
\langle \pi(f*g)\eb^\pi_m,\eb^\pi_n\rangle
&=\langle \pi(g)\eb^\pi_m,\pi(f)^*\eb^\pi_n\rangle
\\
&=\langle \pi(g)\eb^\pi_m,\eb^\pi_n\rangle\, \langle \pi(f)\eb^\pi_n, \eb^\pi_n\rangle
\\
&=\langle \pi(g)\eb^\pi_m,\eb^\pi_n\rangle\, \langle \pi(f')\eb^\pi_m, \eb^\pi_m\rangle 
\\
&=\langle \pi(f')\eb^\pi_m,\pi(g)^*\eb^\pi_n\rangle
\\
&=\langle \pi(g*f')\eb^\pi_m,\eb^\pi_n\rangle ,
\end{align*}
as required.
\end{proof}


\color{black}

Here is our multiplier theorem. 
The reader who is interested in sharp results will notice that the conditions here are stronger than necessary.
\begin{thm}\label{molt1} 
Let $\delta>\alpha+1$ and $\psi \colon \Strip_\delta \cup \Discr\longrightarrow \C$ be  even and  holomorphic in  $\Int \Strip_\delta$.
If
\[
M_\delta (\psi)= \sup_{\var \in \Strip_\delta}\,(1+|\var |)^{3} |  \psi (\var ) |<\infty,
\]
then there exists a unique function $f\in L^1(G,\omega)_m$ such that $\hat f=\psi$, and
\[
\|f\|_{L^1(G,\omega)}\leq  C_m \, \Bigl(M_\delta(\psi)+
\sum_{\var\in \Discr\setminus \Strip_\delta}  \frac{ |\var\,\psi(\var)|}{2|\var|-\alpha-1}\,\,
\Bigr).
\]
\end{thm}

\begin{proof} 
By the inversion formula~\eqref{e:inv2}, 
\[
f(x)
=\frac{1}{8\pi^2}\int_\R\psi(i\lambda)\,\phi_{m,i\lambda}(x)\,|\bfc(i\lambda)|^{-2}
\, d\lambda
+\frac{1}{8\pi}\sum_{\var\in \Discr}|\var|\, \psi(\var)\, \phi_{m,\var}(x)\ .
\]
Let $\chi$ be the characteristic function of $\{x\in G\, :\, \|x\|\le 1\}$; we consider $\chi f$ and $(1-\chi)f$.

Since spherical functions are locally bounded and $|\var\, \psi(\var)|\le M_\delta(\psi)$ if $\var\in \Discr\cap \Strip_\delta$, we deduce from \eqref{stimaHC} that
\beas
 |\chi(x)f(x)|
 & \leq C \,M_\delta(\psi)\, \int_\R \, (1+|\lambda|)^{-2}\,\, d\lambda +\frac{1}{8\pi}\sum_{\var\in \Discr}|\var\, \psi(\var)|
 \\
 &\leq C \Bigl(M_\delta(\psi)+
\sum_{\var\in \Discr\setminus \Strip_\delta} |\var\,\psi(\var)|\,
\Bigr),
\eeas
so $\chi f$ is bounded and $\|f\chi\|_{L^1(G,\omega)}$ satisfies the required norm estimate.

We now estimate $|(1 - \chi) f|$. 
Since $f$ is of type $m$, it suffices to estimate $|f(a_r)|$. 
 
Let  $t\in (\frac{\alpha+1}2,\frac{\delta}{2})\setminus\Discr$. 
The hypotheses  on  $\psi$ allow us to apply Cauchy's theorem as before, and we  obtain
\[
f(a_r)=\frac{1}{4\pi^2} \,\int_\R\psi(t+i\lambda)\,
\mathbf{c}_m(t+i\lambda)^{-1}\, \Phi_{m,-t-i\lambda}(a_r) \, d\lambda
+\frac{1}{8\pi}\sum_{ \var\in \Discr\setminus \Strip_{2t}}|\var|\, \psi(\var)\, \phi_{m,\var}(a_r).
\]

Recall from~\eqref{e:stimaphi} that the spherical functions $\phi_{m,\var}$ with $\var\in \Discr\setminus\Strip_{\alpha+1}$ lie in $L^1(G,\omega)_m$, so the finite sum is easily handled. Moreover,   by   \eqref{stimaHC}  and \eqref{stima_nu},
\begin{align*}
&\left|\int_\R\psi(t+i\lambda)\,
\mathbf{c}_m(t+i\lambda)^{-1}\, \Phi_{m,-t-i\lambda}(a_r) \, d\lambda
\right|
\\&\qquad=
e^{-(1+2t)r}\, 
\left|
\int_\R\psi(t+i\lambda)\,
\mathbf{c}_m(t+i\lambda)^{-1}\, (1+\nu_m(-t-i\lambda,r))\, e^{-2i\lambda r}\, d\lambda
\right|
\\
&\qquad\leq  
C_m \,e^{-(1+2t)r}\, \int_\R|\psi(t+i\lambda)|\,(1+|\lambda|)^{1/2}\, d\lambda
\\
&\qquad\leq 
C_m \,e^{-(1+2t)r}\, M_\delta(\psi).
\end{align*}
The theorem follows by integration in Cartan coordinates:
\begin{align*}
\|f(1-\chi)\|_{L^1(G,\omega)_m}
&\leq C_m \, M_\delta(\psi)\,\int_1^\infty e^{-(1+2t)r}\, \sinh (2r)\, e^{\alpha r}\, dr
\\
&\leq C_m \, M_\delta(\psi)\,\int_1^\infty e^{-(2t-\alpha-1)r}\, dr,
\end{align*}
as required.
\end{proof}
 
\color{black}
\section{Spectral theory and the spherical transformation}

Let $\Omega_m$ be the restriction of the differential operator $\Omega$ to $L^2(G)_m$.
We take the weight $\omega$ to be $\omega_\alpha$, where $\alpha \in [0,\infty)$, and define   
$\alpha_m$ to be an arbitrary real number strictly greater than 
$\tfrac14 \max \left\{ \alpha^2 +2\alpha , m^2 -2 |m| \right\}$.

\begin{thm}\label{thm:heat-kernels}
There exists a ``heat semigroup'' $(h_t)_{t\in\R^+}$ of functions  in $L^1(G,\omega)_m$ that satisfies the conditions
\[
h_s * h_t = h_{s+t} 
\qquad\text{and}\qquad
\| h_t \|_{L^1(G,\omega)}  \lesssim e^{\alpha_m t} 
\]
for all $s, t \in \R^+$, and, setting $u(\cdot,t) = f * h_t$ for $f \in L^2(G)_m$,
\[
\frac{\partial}{\partial t} u(\cdot,t) + \Omega_m u(\cdot,t) = 0
\qquad\text{and}\qquad
u(\cdot,t) \to f \quad\text{as $t \to 0$}.
\]
\end{thm}

\begin{proof}
We define $h_t \in L^2(G)_m$ by its spherical transform: 
\[
\hat h_t (\var) := \exp( t ( \var^2 - \tfrac{1}{4}) ) 
\qquad\forall \var \in \C . 
\]
If $\var = u + iv$, then $| \hat h_t(u+iv) | = \exp (t(u^2 - v^2 -1/4 ))$, so $\hat h_t$ is bounded  and decays exponentially as we approach infinity in every strip $\Strip_\delta$.
By Theorem~\ref{molt1},  $h_t \in L^1(G,\omega)$.
The semigroup property and the differential equation follow from Fourier analysis.

It now suffices to show the norm estimate, for once this is established, the convergence of $u(\cdot ,t)$ as $t \to 0$ also follows.
The norm estimate follows from Theorem \ref{molt1}:

\begin{align*}
\| h_t \|_{L^1(G,\omega)_m} 
&\leq C \Bigl( \sup_{\var \in \Strip_\delta}\,(1+|\var |)^{3} |  e^{t(\var^2 - 1/4) } |
+  \sum_{\var\in \Discr\setminus \Strip_\delta} \frac{|\var|  e^{t(\var^2 - 1/4) }}{2|\var|-\alpha -1}\Bigr) \\
&\leq C' \Bigl(    e^{t(\delta^2 - 1)/4 } 
+  \max_{\var\in \Discr\setminus \Strip_\delta}   | e^{t(\var^2 - 1/4) } | \Bigr) ,
\end{align*}
as claimed.
\end{proof}

By improving the multiplier theorem, we could improve the $e^{\alpha_m t}$ term in Theorem \ref{thm:heat-kernels} when $|m| $ is small to a polynomial term, but we do not bother to do so here.

We define $\beta_m$ to be $( \alpha_m + 1/4)^{1/2}$ and $\gamma_m$ to be $( \alpha_m + 5/4)^{1/2}$.

We are going to use Fourier analysis, and it is convenient to work in the Fourier variable.
For $z \in \C\setminus \Strip_{2 \beta_m}$, we define the function $r_z$ in $L^2(G)_m$ by 
\begin{equation}\label{eq:def-rz}
\hat r_z(\var) := (z^2  - \var^2)^{-1}
\qquad\forall s \in \Strip_{2\beta_m} .
\end{equation}

\begin{lem}\label{lem:resolvents}
The mapping $z \mapsto r_z$ of \eqref{eq:def-rz} is continuous from $\C \setminus \Int \Strip_{2\gamma_m}$  to $L^1(G,\omega)_m$, 
\begin{equation}\label{eq:real-z-estimate}
\norm{ r_\zeta }_{L^1(G,\omega)} \lesssim (\zeta^2  - \beta_m^2)^{-1} 
\qquad\forall \zeta \in \R \setminus [-\beta_m,\beta_m],
\end{equation}
and
\begin{equation}\label{eq:all-z-estimate}
\norm{ r_z }_{L^1(G,\omega)} \lesssim (1 + |z|)^4         
\qquad\forall z \in \C \setminus \Int \Strip_{2\gamma_m}.
\end{equation}
Further, for all $N \in \N$, 
\[
\lim_{\zeta \in \R,\ |\zeta| \to \infty} \norm{ (\zeta^2 r_\zeta)^{*N} * f - f }_{L^1(G,\omega)} \to 0
\qquad\forall f \in L^1(G,\omega)_m ,
\]
where $f^{*N}$ denotes the $N$-fold convolution power of $f$.
\end{lem}

\begin{proof}
We show first that, if $\abs{\Re z} \geq \gamma_m$ and $\abs{\Re \zeta} \geq \gamma_m$, then $r_z - r_\zeta  \in L^1(G,\omega)_m$ and 
\begin{equation}\label{eq:step-1}
\norm{ r_z - r_\zeta }_{L^1(G,\omega)} \lesssim |\zeta^2 - z^2| (1 + |z|) ^2(1 + |\zeta|) ^2.
\end{equation}
Then we shall prove \eqref{eq:real-z-estimate}.
Once we have done so, we may take $\zeta$ in \eqref{eq:step-1} to be $\gamma_m$, and \eqref{eq:all-z-estimate} follows immediately.

To prove \eqref{eq:step-1}, we apply the multiplier theorem.
By definition, 
\[
\hat r_z(\var) - \hat r_\zeta(\var) 
= \frac{1}{ z^2  - \var^2} - \frac{1}{ \zeta^2  - \var^2}
= \frac{\zeta^2 - z^2}{ (z^2  - \var^2) (\zeta^2  - \var^2)}  \,.
\]
There is a constant $C$ such that, if  $\abs{\Re \var} \leq \beta_m$ and $\abs{\Re z}  \geq \gamma_m$, then
\[
\left| \frac{1}{ z  - \var} \right| 
\leq C \frac{1 + |z| }{ 1  + |\var|}  \,,
\]
whence
\[
\left| \frac{1}{ (z^2  - \var^2) (\zeta^2  - \var^2)}\right| 
\leq C^4   \frac{(1 + |z|)^2 (1 + |\zeta|)^2 }{ (1  + |\var|)^4 }  \,.
\]
Now \eqref{eq:step-1} follows from the multiplier theorem.

To prove \eqref{eq:real-z-estimate}, we treat the operator $\Omega_m$, which is self-adjoint on $L^2(G)_m$ with spectrum contained in $(-\infty, \alpha_m]$, by using spectral theory.
When $\lambda > \alpha_m$,  the resolvent operator $R_\lambda$ may be defined using the formula
\[
R_\lambda
:= \int_{0}^{\infty}  e^{-\lambda t} e^{t \Omega_m} \,dt 
= (\lambda - \Omega_m)^{-1}.
\]
Then the convolution kernel $\ell_\lambda$ of $R_\lambda$ is given by
\[
\ell_\lambda
= \int_{0}^{\infty}  e^{-\lambda t} h_t \,dt .
\]
By Theorem \ref{thm:heat-kernels}, the heat kernels $h_t$ lie in $L^1(G,\omega)_m$ and $\norm{ h_t }_{L^1(G,\omega)} \lesssim e^{\alpha_m t}$ for all $t \in \R^+$.
It follows that $\ell_\lambda \in L^1(G,\omega)_m$ and 
\[
\norm{\ell_\lambda}_{L^1(G,\omega)} 
\lesssim \int_{0}^{\infty}  e^{-\lambda t} e^{\alpha_m t} \,dt 
= \frac{1}{\lambda - \alpha_m } \,
\]
for all $\lambda \in (\alpha_m , \infty)$.
Taking spherical transforms, we see that 
\[
\hat \ell_\lambda (\var)
= \int_{0}^{\infty}  e^{-\lambda t} \hat h_t (\var) \,dt 
= \int_{0}^{\infty}  e^{-\lambda t} \exp( t ( \var^2 - 1/4) )  \,dt 
= (\lambda + \tfrac{1}{4} - \var^2)^{-1}. 
\]
Comparing this formula with \eqref{eq:def-rz}, we see that $r_z = \ell_\lambda$ when $z^2 = \lambda + \tfrac{1}{4}$.
Consequently  $r_z \in L^1(G,\omega)_m$ if $z \in \R \setminus [-\beta_m,\beta_m]$, and
\[
\norm{r_z}_{L^1(G,\omega)} \lesssim \frac{1}{z^2  - \beta_m^2} \,. 
\]

The last part of the lemma is easy. 
Suppose that $f$ is in $C^\infty_c(G)_m$, so that $\hat f$ is entire and has rapid decay at infinity in every strip $S_\delta$.
Since
\[
\lpar (\zeta^2 r_\zeta)^{*N} * f - f \rpar\afterhat(\var) = 
\lpar \frac{1}{1 - \var^2/\zeta^2}\rpar^N \hat f(\var) - \hat f(\var) ,
\]
the multiplier theorem (Theorem~\ref{molt1}) shows that
\[
\lim_{\zeta \to +\infty}\norm{  (\zeta^2 r_\zeta)^{*N} * f - f }_{L^1(G,\omega)} =0 .
\]  
Since the norms $\norm { (\zeta^2r_\zeta)^{*N }}_{L^1(G,\omega)}$ are uniformly bounded   for $\zeta \in \R \setminus [-\gamma_m,\gamma_m]$, this extends to all $f \in L^1(G,\omega)_m$ by density. 
\end{proof}
  

\begin{thm}\label{thm:generator}
There exists $f_m \in L^1(G,\omega)_{m}$ such that the closed subalgebra of $L^1(G,\omega)_{m}$ generated by $f_m$ is all $L^1(G,\omega)_{m}$.
\end{thm}

\begin{proof}
Let $A$ be the closed subalgebra of $L^1(G,\omega)_m$ generated by the resolvent function $r_{\gamma_m}$; we are going to show that $A = L^1(G,\omega)_m$ and take $f_m$ to be $r_{\gamma_{m}}$.

Recall that $\cL(\Ban)$ denotes the set of bounded linear operators on the Banach space $\Ban$.

There is a standard argument that the resolvent depends analytically on the parameter and the resolvent set is open:
if $R_{\lambda_0}\in \cL(L^1(G,\omega)_m)$, and $|\lambda -\lambda_0| < 1/\norm{R_{\lambda_0}}_{\cL(L^1(G,\omega)_m)}$, then
\begin{align*} 
R_\lambda 
&= \frac{1}{\lambda - \Omega_m} 
= \sum_{n=0}^{\infty} \frac{(\lambda_0-\lambda)^n}{(\lambda_0 - \Omega_m)^{n+1}} 
= \sum_{n=0}^{\infty} (\lambda_0-\lambda)^n R^{n+1}_{\lambda_0},
\end{align*}
and this last series converges absolutely in $\cL(L^1(G,\omega)_m)$, so $R_\lambda \in \cL(L^1(G,\omega)_m)$.
It also follows that when the resolvent operator $R_{\lambda_0}$ is given by convolution with an element of $L^1(G,\omega)_m$, so is $R_\lambda$, and when $R_{\lambda_0}$ is given by convolution with an element of $A$, so is $R_\lambda$.
Then the set $S := \{ z \in \C \setminus \Int \Strip_{2\gamma_m} : r_z \in A \}$ is open in $\C \setminus \Int\Strip_{2\gamma_m}$.
The set $S$ is also closed in $\C \setminus \Int \Strip_{2\gamma_m}$,  since $r_{z_j} \to r_z$ in $L^1(G,\omega)_m$ if $z_j \to z \in \C \setminus \Int \Strip_{2\gamma_m}$ by Lemma \ref{lem:resolvents}, and is symmetric about $0$; hence $S = \C \setminus \Int\Strip_{2\gamma_m}$.

Take $f \in L^1(G,\omega)_m$; we must show that $f \in A$.  
From the Cartan decomposition, $f$ may be approximated arbitrarily closely in $L^1(G,\omega)_m$ by elements $\funz_n$ with compact support, and so it suffices to show that $\funz_n \in A$.
By \eqref{stimastriscia} $\hat \funz_n$ is uniformly bounded on $\Strip_{2\gamma_m}$.
By Lemma \ref{lem:resolvents}, $\funz_n$ can be approximated arbitrarily well in $L^1(G,\omega)_m$ by elements of the form $(\zeta^2 r_\zeta)^{*4}* \funz_n$ by taking $\zeta$ large in $\R$, so it suffices to treat $(\zeta^2 r_\zeta)^{*4} *\funz_n$; it is evident that  $((\zeta^2 r_\zeta)^{*4}* \funz_n)\afterhat (\var)$ vanishes like $(1+|\var|)^{-8}$ at infinity in the strip $\Strip_{2\beta_m}$.
Let $\funzz=(\zeta^2 r_\zeta)^{*4}* \funz_n$. If we can show that $\funzz\in A$, then $f\in A$ by a limiting argument.

The Cauchy integral formula and the evenness of $\hat \funzz$ imply that
\[
\begin{aligned}
\hat \funzz(\var)  
&= \frac{1}{2\pi i} \int_{\Gamma} \hat \funzz(z) \left(\frac{1}{ z - \var} + \frac{1}{z + \var }\right) \,dz \\
&= \frac{1}{\pi i} \int_\Gamma z \hat \funzz(z) \frac{1}{z^2 - \var^2 }  \,dz \\
&= \frac{1}{\pi i} \int_\Gamma z \hat \funzz(z) \, \hat r_{z}(\var)  \,dz 
\qquad\forall s \in \Int \Strip_{2\beta_m};
\end{aligned}
\]
here $\Gamma$ is the contour given by $z = \gamma_m +iy$ where $y$ varies from $-\infty$ to $\infty$ in $\R$.
Hence
\[
\funzz=\frac{1}{\pi i} \int_\Gamma z \hat \funzz(z) \,  r_{z}  \,dz ,
\]
and since $r_z \in A$ and $\norm{r_z}_{L^1(G,\omega)} \lesssim (1 + |\Im z|)^4$ while $|\hat \funzz(z)| \lesssim (1 + |\Im z|)^{-8}$, the last integral
converges as an $A$-valued Bochner integral, and hence  $\funzz \in A$, as required.
\end{proof}

\section{Admissibility of irreducible  representations of $\mathrm{SL}(2,\R)$}

In this section we apply the previous results to study the admissibility of irreducible Banach representations of $G$. We shall prove that an irreducible   representation $\rep$ of $G$ on a Banach $\Ban$ is admissible if one of the nontrivial  Banach subspaces $\Ban_m$ has the ISP.
Coupled with Soergel's result \cite{Soergel88}, this becomes an equivalence.

\begin{thm}
Let $\rep$ be an irreducible representation of $G$ on a Banach space $\Ban$.
Assume that the Banach subspace $\Ban_m$ is nontrivial and has the invariant subspace property for some $m \in \Z$.
Then $\dim(\Ban_{n}) \leq1$ for all $n \in \Z$.
\end{thm}

\begin{proof}
By Lemma  \ref{lem:mod-growth}, there exists $\alpha \in [0,\infty)$ such that 
\[
\| \rho(x) \|_{\mathcal{L}(\Ban)} \leq C \|x\|^\alpha := C \omega(x) 
\qquad\forall x \in G.
\]

With a view to reaching a contradiction, suppose that  $\dim(\Ban_m) > 1$.
By Lemma \ref{lem:irred-on-L1m}, the restriction to $L^1(G,\omega)_m$ of the integrated representation of $L^1(G,\omega)$ acts irreducibly on $\Ban_m$.
By Theorem \ref{thm:generator}, there exists $f_m \in L^1(G,\omega)_m$ that generates $L^1(G,\omega)_m$. 
Let $\Ban_m^\circ$ be a nontrivial closed invariant subspace for $\rep(f_m)$ acting on $\Ban_m$.
Then $\Ban_m^\circ$ is a nontrivial closed invariant subspace for $\rho(L^1(G,\omega)_m)$ on $\Ban_m$, which is our desired contradiction.

Thus $\dim(\Ban_m) = 1$; take $\xi \in \Ban$ such that $\Ban_m = \C \xi$.

Take $n \in \Z \setminus \{m\}$ such that $\Ban_n$ is nontrivial, and $\eta \in \Ban_n \setminus\{0\}$.
Since $\rho$ is irreducible $\rep( C^\infty_c(G))\xi $ is dense in $\Ban$, so there exists $\{g_j\}_j \in C^\infty_c(G)$ such that $\norm{ \rep(g_j) \xi - \eta}_{\Ban}\longrightarrow 0$ as $ j\to \infty$.
Now since $\mathcal{P}_n$ is a bounded projection,
\[
\norm{ \rep(\mu_n * g_j *\mu_m) \xi - \eta  }_{\Ban}
= \norm{ \mathcal{P}_n \rep(g_j) \mathcal{P}_m\xi - \eta }_{\Ban}
= \norm{ \mathcal{P}_n ( \rep(g_j) \xi - \eta ) }_{\Ban}  \longrightarrow 0
\]
as $ j\to \infty$,
so there exists $g\in C^\infty_c(G)$ with the invariance property \eqref{eq:invariance-mn} such that  $\rho(g )\xi\in \Ban_n\setminus\{0\}$. 
By Lemma~\ref{lem:irred-on-L1m} applied to $\Ban_n$, the space $\rep( C^\infty_c(G)_n * g) \xi$ is dense in $\Ban_n$. Moreover, from Theorem \ref{thm:f-star-g}, $C^\infty_c(G)_n * g=g *C^\infty_c(G)_m$. Since $\Ban_m$ is one dimensional,  $\rep(C^\infty_c(G)_m)\xi=\C\xi$. 
This shows that 
$\C\rep(g  )\xi$ is dense in $\Ban_n$ and hence $\dim(\Ban_n) =1$.
\end{proof} 
\color{black}

\begin{cor}
Assume that Hilbert spaces have the ISP. 
Then each irreducible uniformly bounded representation $\rho$ of $G$ on a Hilbert space is equivalent to one of the uniformly bounded representations of Ehrenpreis and Mautner \cite{EMI, EMII, EMIII} and Kunze and Stein \cite{KS60}.
\end{cor}

\begin{proof}
Our theorem implies that $\rho$ is admissible.
The classification of admissible representations (see, for instance, \cite[Chapter 2]{HT})  then implies the affirmation.
\end{proof}

\section{Final remarks}
We have seen that every irreducible Hilbert space representation of $G$ is admissible, modulo the solution to the invariant subspace problem.
Conversely, by the embedding theorem of Delorme and Souafi  \cite{DS}, every irreducible Harish-Chandra module of every seimsimple Lie group can be globalised to a  Hilbert representation.
Our arguments may extend to other semisimple Lie groups. 
However, the extension seems to require a version of the invariant subspace problem that deals with families of operators, which does not seem to have been considered.

\section{Acknowledgements}
The authors thank Bernhard Kr\"otz for helpful comments.

\end{document}